\documentclass[12pt,oneside]{amsart}

      \usepackage{amssymb}

      \theoremstyle{plain}
      \newtheorem{theorem}{Theorem}[section]
      \newtheorem{lemma}[theorem]{Lemma}
      
      \newtheorem{prop}{Proposition}
      \theoremstyle{definition}
      
      \theoremstyle{conj}
      \newtheorem{conj}[theorem]{Conjecture}
      \theoremstyle{remark}
      
      \makeatletter
      \let\runauthor\@author
      \let\runtitle\@title
      \makeatother
      
\begin{document}

%



   \author{Masoud Khalkhali and Sajad Sadeghi
\vspace{0.5cm}\\ Department of Mathematics,  University of Western Ontario, 
 London, Ontario, Canada
}
 \date{}


   \email{masoud@uwo.ca, ssadegh3@uwo.ca}
   

   \title{On  Logarithmic Sobolev Inequality for the Noncommutative Two Torus 
     }


   \begin{abstract}
     An analogue of Gross' logarithmic Sobolev inequality for a class of elements of noncommutative two tori is proved.
   \end{abstract}

   \date{}


   \maketitle

   \section{Introduction}
The subject of logarithmic Sobolev inequality has its roots in the paper of E. Nelson \cite{N}, where he proved the contractivity of the semi-group generated by the Gauss-Dirichlet form operator. Then, shortly after that L. Gross introduced logarithmic Sobolev inequalities in \cite{G} and using them gave a different proof of the contractivity of the semi-group generated by Gauss-Dirichlet form operator.

Let $\nu$ be the Gauss measure on $\mathbb{R}^{n}$ and   
 $$N:D(N) \subseteq L^{2}(\mathbb{R}^{n},\nu)\longrightarrow L^{2}(\mathbb{R}^{n},\nu), $$ be the Gauss-Dirichlet form operator defined by
$$
\left \langle Nf,g \right \rangle=\int_{\mathbb{R}^{n}}
 \bigtriangledown f(x) \cdot \bigtriangledown g(x)d\nu(x)
$$
where $
\left \langle Nf,g \right \rangle =\int_{\mathbb{R}^{n}} f(x) \overline{g(x)} d\nu(x)$,
and $\bigtriangledown f$ is the weak gradient of $f$.
Nelson showed that for $1<q\leq p< \infty$, if 
$$
e^{-2t}\leq\dfrac{(q-1)}{(p-1)},
$$ 
then $$\Vert e^{-tN} \Vert _{q\rightarrow p}\leq1,$$
where
$$\Vert e^{-tN} \Vert _{q\rightarrow p}=\text{sup} \left\lbrace  \Vert e^{-tN}f\Vert_{p}: f\in L^{2}(\nu)\cap L^{q}(\nu), \Vert f \Vert _{q} \leqslant 1 \right\rbrace. $$
This means that for $$
t\geqslant \ln \sqrt{\frac{p-1}{q-1}},
$$ 
$e^{-tN}$ is a contraction from $L^{q}(\mathbb{R}^{n},\nu)$ to $L^{p}(\mathbb{R}^{n},\nu)$. He even proved more. Indeed, he showed that $e^{-tN}$ is an unbounded operator from $L^{q}(\mathbb{R}^{n},\nu)$ to $L^{p}(\mathbb{R}^{n},\nu)$, if 
$$
t<  \ln \sqrt{\frac{p-1}{q-1}}.
$$

The classical Sobolev inequality states that for $f \in C^{\infty}_{c}(\mathbb{R}^{n})$,
\begin{equation} \label{clas}
\Vert f \Vert _{L^{q}(\mathbb{R}^{n}, dx)}\leq C_{p,n} \Vert \left |  \bigtriangledown f \right | \Vert _{L^{p}(\mathbb{R}^{n}, dx)}
\end{equation}
where $1\leq p < \infty$, $\dfrac{1}{q}= \dfrac{1}{p}- \dfrac{1}{n}$,  $dx$ is the Lebesgue measure and $C_{p,n}$ is a constant depending only on $n$ and $p$. So (\ref{clas}) implies that if the gradient of the function $f$ is in $L^{p}(\mathbb{R}^{n}, dx)$, then $f$ must be in $L^{q}(\mathbb{R}^{n}, dx)$. These inequalities strongly depend on the dimension of $\mathbb{R}^{n}$.

In \cite{G}, Gross proved a logarithmic Sobolev inequality
\begin{equation} \label{GRls}
\int_{\mathbb{R}^{n}} \left | f(x) \right | ^{2} \ln \left | f(x) \right |
d\nu (x)\leq \int_{\mathbb{R}^{n}}
\left | \bigtriangledown f(x) \right | ^{2} d\nu(x)+ \Vert f \Vert _{2}^{2} \ln \Vert f \Vert _{2}
\end{equation}
for $f \in L^{2}(\mathbb{R}^{n},\nu)$, and showed that this inequality is equivalent to Nelson's result of contractivity that we just mentioned. 

Unlike the classical Sobolev inequality, Gross' logarithmic Sobolev inequality is dimension independent. Using (\ref{GRls}) and the Poincar\'{e} inequality we see that if the gradient of the function $f$ is in $L^{2}(\mathbb{R}^{n},\nu)$, then $f$ is in the Orlicz space $L^{2} \ln  L( \nu) $. This somehow justifies the name logarithmic Sobolev inequality. Gross also derived in \cite{G} a weaker version of (\ref{GRls}), from Nirenberg's form of classical Sobolev inequality \cite{NI}. This  version, not surprisingly depends on the dimension.

 Since then people have given various proofs of logarithmic Sobolev inequalities by different methods: O. Rothaus \cite{ROT} has proved it using Jensen's inequality and the positivity of the lowest eigenfunction for a Sturm-Liouville boundary value problem with Dirichlet boundary conditions. R. A. Adams and F. H. Clarke also have given a simple proof based on calculus of variations \cite{ADCL}. 

One can replace $(\mathbb{R}^{n}, \nu)$ with a probability space $(X, \mu)$ and the Gauss-Dirichlet form with a densely defined positive quadratic form on $L^{2}(X,\mu)$, say $\mathcal{E}$. Then we say the logarithmic
Sobolev inequality holds for $\mathcal{E}$, if for $f \in Dom(\mathcal{E})$, 
$$
\int_{X} \left | f(x) \right | ^{2} \ln \left | f(x) \right |
d\mu (x)\leq \mathcal{E}(f,f)+ \Vert f \Vert _{2}^{2} \ln \Vert f \Vert _{2}.
$$
This way one can talk about logarithmic Sobolev inequalites on  Riemannian manifolds \cite{Path}.

F. Weissler has proved in \cite{W} a logarithmic Sobolev inequality 
on the circle. Indeed, using Fourier series he has shown that for a positive function $f$ in $L^{2}(\mathbb{T, \mu})$, where $\mathbb{T}$ is the unit circle and $\mu$ is the normalized Lebesgue measure, if 
$$f(\theta)= \sum_{n=-\infty}^{\infty} a_{n}e^{in \theta},$$
then 
\begin{equation}\label{wwe}
\int_{\mathbb{T}} f \log f d\mu \leq \sum_{n=-\infty}^{\infty} \vert n \vert \vert a_{n} \vert ^{2}+ \Vert f \Vert _{2}^{2} \log \Vert f \Vert _{2}. 
\end{equation}
Since $$ \sum_{n=-\infty}^{\infty} \vert n \vert  \vert a_{n} \vert ^{2} \leq \sum_{n=-\infty}^{\infty} \vert n \vert ^{2} \vert a_{n} \vert ^{2}=\Vert  \bigtriangledown f  \Vert ^{2}_{2}=\int_{\mathbb{T}} \left |\bigtriangledown f \right |^{2} d\mu,  $$
Weissler's result is even stronger than the original logarithmic Sobolev inequality
$$
\int_{\mathbb{T}} f \log f d\mu \leq \int_{\mathbb{T}} \left |\bigtriangledown f \right |^{2} d\mu+ \Vert f \Vert _{2}^{2} \log \Vert f \Vert _{2}.
$$

There is a useful survey of related topics and applications of logarithmic Sobolev inequalities by Gross in \cite{GR2}. One can also find more references therein.

Since the introduction of noncommutative geometry by Alain Connes in \cite{con} (see also \cite{con1}), noncommutative tori have proved to be an invaluable tool to understand and test many aspects 
of noncommutative geometry that is  not present in the commutative case.  The results are simply too many to be cited here. The present paper should be seen as a step in understanding  aspects of measure theory and analysis on noncomutative tori that have been largely untouched so far. The combinatorial challenges one faces in extending the logarithmic Sobolev inequality, at least in the form that we understand it,  seemd to us as a very interesting problem by itself.

Let $\theta \in \mathbb{R} \setminus \mathbb{Q}$. The universal C*-algebra generated by two unitaries $U,V$ such that $UV=e^{2\pi i \theta}VU$, is  called  the  irrational rotation algebra and is denoted by $A_{\theta}$. It is a simple algebra and has a unique positive faithful normalized trace $\tau$. $A_{\theta}$ is the noncommutative deformation of $C(\mathbb{T}^{2})$, the algebra of continuous functions on the 2-torus. More details about $A_{\theta}$ can be foun in \cite{con1, MK, Rif}. Let
$$
A_{\theta}^{\infty}=\left\lbrace 
\underset{(m,n)\in \mathbb{Z}^{2}}{\sum}  
 a_{m,n} U^{m} V^{n}: a_{m,n}  \text{ is rapidly decreasing} \right\rbrace.
$$
By rapidly decreasing we mean for all $k \in \mathbb{N}$,
$$
\underset{(m,n) \in \mathbb{Z}^{2}}{\text{Sup}} (1+m^{2}+n^{2})^{k} \vert a_{m,n} \vert ^{2} < \infty.
$$
$A_{\theta}^{\infty}$ is a dense subalgebra of the irrational
 rotation algebra and it is the analogue of $C^{\infty}(\mathbb{T}^{2})$, the algebra of smooth functions on the 2-torus. $A_{\theta}^{\infty}$ is
  called the noncommutative two torus. 
  
 Let $a= \underset{(m,n)\in \mathbb{Z}^{2}}{\sum}  
 a_{m,n} U^{m} V^{n} $ be in $A_{\theta}^{\infty}$. Then $$a^{*}= \underset{(m,n)\in \mathbb{Z}^{2}}{\sum}  
 \overline{a_{m,n}} V^{-n} U^{-m}$$
 $$=\underset{(m,n)\in \mathbb{Z}^{2}}{\sum}  
 \overline{a_{m,n}} \mathit{e}^{-2\pi imn \theta } U^{-m} V^{-n}$$
 $$= 
  \underset{(m,n)\in \mathbb{Z}^{2}}{\sum}  
 \overline{a_{-m,-n}} \mathit{e}^{-2\pi imn \theta } U^{m} V^{n}. $$
 So if $a=a^{*}$, we have 
\begin{equation} \label{sa}
a_{m,n}=\overline{a_{-m,-n}} \mathit{e}^{-2\pi imn \theta }.
\end{equation}
Moreover, if $b= \underset{(m,n)\in \mathbb{Z}^{2}}{\sum}  
 b_{m,n} U^{m} V^{n} \in A_{\theta}^{\infty}, $
we have 
$$ab=\underset{(m,n)\in \mathbb{Z}^{2}}{\sum}  
 c_{p,q} U^{p} V^{q},$$
 where \begin{equation} \label{pr}
  c_{p,q}= \underset{(m,n)\in \mathbb{Z}^{2}}{\sum}  
 a_{m,n} b_{p-m,q-n} \mathit{e}^{-2\pi i(p-m)n \theta } 
 \end{equation}    
    The unique trace $\tau$ on $A_{\theta}$ which plays
     the role of integration in the noncommutative setting, extracts the
     constant term of the elements of
     $A_{\theta}^{\infty}$, i.e. $\tau (a)=a_{0,0}$.
This trace can be used to define an $L^{2}$-norm on $A_{\theta}^{\infty}$ by
$$\Vert a \Vert_{2}^{2}:= \tau(a^{*}a) .$$
 Using (\ref{sa}) and (\ref{pr}) one can show $\Vert a \Vert_{2}^{2}= \underset{(m,n)\in \mathbb{Z}^{2}}{\sum} \vert a_{m,n} \vert ^{2}$. One can also define Sobolev norms on $A_{\theta}^{\infty}$. For more details see \cite{R}, where 
J. Rosenberg has developed the Sobolev theory on the noncommutative two torus. 
 
 In this paper we will use Weissler's method \cite{W} to prove a logarithmic Sobolev inequality for a class of elements of 
 the noncommutative two torus. 
 In section 2 we will prove some lemmas that we will need later on. In section 3 we will first state 
 our conjecture about a logarithmic Sobolev inequality on the noncommutative 2-torus and then 
 we will prove that conjecture for a class of elements of the noncommutative 2-torus. This would be the main result of this paper. Although we have not been able to prove 
 the logarithmic Sobolev inequality for arbitrary positive 
 elements, we think the inequality must hold for those elements as well. In section 4 we will try to generalize the proof of the main result to prove the conjecture, but in the middle of the way we will see that we will face a problem. We hope that we can bypass this problem in a follow up paper. 
\section{ Preliminaries}
   \label{prelim}
 In this section we will prove some technical lemmas that will be needed later on.  
\begin{lemma} \label{ana}
Let $G$ be an analytic function in some complex neghiborhood of the interval $[0,1]$. Suppose all the coefficients in the power series expansion of $G$ around $r=0$ are nonnegative. Then $G(1)\geq0$.
\end{lemma}
\begin{proof}
First we show that we can find 
finitely many points $x_{0}=0< x_{1}< x_{2}< \cdots < x_{n}=1$ in $[0,1]$ and 
finitely many discs $D_{0}, D_{1},D_{2}, \cdots,$ $D_{n}$ such that $x_{i}$ for $i=0,1, \cdots, n$,  is the 
center of $D_{i}$, $G$ has a power series expansion around 
$x_{i}$ on $D_{i}$ and $x_{i} \in D_{i-1}$, for $i=1,2, \cdots, n$. 
To show this let N be the open set in $\mathbb{C}$ containing $[0,1]$ on which $G$ is analytic.  Define
$$ 
F:[0,1] \longrightarrow \mathbb{R}^{>0}
$$
by sending $r \mapsto \text{dist}(r, N^{c})$.
$F$ is a 
continuous function on a compact set, so it attains its minimum. 
Let $\delta$ be the minimum of $F$ and 
$x_{0}=0, x_{1}=\dfrac{\delta}{2}, x_{2}=\delta, 
x_{3}=\dfrac{3\delta}{2}, \cdots, x_{n-1}=
\dfrac{(n-1) \delta}{2}, x_{n}=1$, where 
$n=\lfloor \dfrac{2}{\delta} \rfloor +1$. For 
$i=0,1, \cdots, n$, let $R_{i}$ be the radius of convergence of the power series
 expansion of $G$ around $x_{i}$, and $D_{i}$ be the disc centred at $x_{i}$ with radius $R_{i}$. For $i=0,1, \cdots, n$, we have $\dfrac{\delta}{2}< R_{i}$. So $x_{i} \in D_{i-1}$, for $i=1,2, \cdots, n$.

Let 
\begin{equation} \label{zero}
G(z)=\sum_{k=0}^{\infty} \dfrac{G^{(k)}(0)}{k!} z^{k}
\end{equation}
be the power series expansion of $G$ around $0$ on $D_{0}$. Since $x_{1} \in D_{0} $, we plug $x_{1}$ in (\ref{zero})for $k\geq0 $, $G^{(k)}(0)\geq0$, 
$$G(x_{1})=\sum_{k=0}^{\infty} \dfrac{G^{(k)}(0)}{k!} x_{1}^{k}\geq0$$ 
If we substitute $x_{1}$ into the derivative of (\ref{zero}), we will get
$$
G^{(1)}(x_{1})=\sum_{k=1}^{\infty} \dfrac{G^{(k)}(0)}{(k-1)!} x_{1}^{k-1}
$$
which is non-negative by the same reason. Differentiating (\ref{zero}) repeatedly, we can show that the derivatives of $G$ at $x_{1}$ which form the coefficients of the power series expansion of $G$ around $x_{1}$ on $D_{1}$ are nonnegative.

Repeating this argument, we can show that all derivatives of $G$ at each $x_{i}$ and in particular at $x_{n}=1$ are non-negative. So $G(1)\geq 0$.
\end{proof}
We will need the following standard and elementary result of  spectral theory in  C*-algebras.
\begin{prop} \label{Rudin}
Let $A$ be a C*-algebra, $x \in A$ and $N$ an open subset of $\mathbb{C}$ containing the spectrum of $x$, $\sigma(x)$. Then there exists $\delta > 0$, such that for $y \in A$, $\Vert y-x \Vert < \delta $ implies $\sigma(y) \subseteq N $.
\end{prop}
\begin{proof}
See \cite{Rud} Theorem 10.20.
\end{proof}
The following proposition will be needed in the proof of the main result of this paper. 
\begin{prop} \label{comp}
Let $a=\underset{\underset{(m,n)\neq (0,0)}{(m,n)\in \mathbb{Z} ^{2} }}{\sum} a_{m,n}  U^{m} V^{n}$ be  
in $A_{\theta}^{\infty}$, such that $a>0$, $a_{0}=1$ and at most finitely many number of $a_{m,n}$'s are nonzero. For $r \in \mathbb{C}$, we put $$x_{r}=\underset{\underset{(m,n)\neq (0,0)}{(m,n)\in \mathbb{Z}^{2}}}{\sum} a_{m,n} r^{(\vert m\vert +\vert n\vert)} U^{m} V^{n}$$
and $P_{r}(a)=1+x_{r}$. Then there is an open neighbourhood $W$ of $[0,1]$ in $\mathbb{C}$, such that for all $r$ in $W$, $\log P_{r}(a)$ can be defined.
\end{prop}
\begin{proof}
Since $a$ is self-adjoint, using (\ref{sa}) for real $r$ we have 
$$a_{m,n} r^{(\vert m\vert +\vert n\vert)}=\overline{a_{-m,-n}} r^{(\vert -m\vert +\vert -n\vert)} \mathit{e}^{-2\pi imn \theta }.$$
So $P_{r}(a)$ is self-adjoint for real $r$, and consequently the spectrum of $P_{r}(a)$ is real for real $r$. Now we show for $0\leq r\leq 1$, $P_{r}(a)$ is a  strictly positive element. Suppose for some $0\leq r\leq 1$, $P_{r}(a)$ is not strictly positive. Let $[t_{0}, t_{1}]$ be the smallest closed interval containing the spectrum of $P_{r}(a)$. We know that there exists a state $\phi$ of $A_{\theta}$, such that $\phi (P_{r}(a))=t_{0}
\leq 0$. Now let
$$ M=\lbrace (m,n) \in \mathbb{Z}^{2}: \, (m,n)\neq 0,a_{m,n} \neq0 \rbrace,$$
$$M_{1}=\lbrace (m,n) \in M : \, m\geqslant0, n\geqslant0 \rbrace,$$
$$M_{2}=\lbrace (m,n) \in M : \, m>0, n<0 \rbrace.$$

Then since $a$ is self-adjoint, using (\ref{sa}) we have
\begin{equation} \label{a=}
a= 1+\underset{(m,n)\in M}{\sum}  
 a_{m,n} U^{m} V^{n}
\end{equation}
$$
=1+ \underset{(m,n)\in M_{1} \bigcup M_{2}}{\sum}  
 a_{m,n} U^{m} V^{n}+\underset{(-m,-n)\in M_{1} \bigcup M_{2}}{\sum}  
 a_{m,n} U^{m} V^{n}
$$
$$
=1+ \underset{(m,n)\in M_{1} \bigcup M_{2}}{\sum}  
 a_{m,n} U^{m} V^{n}+\underset{(m,n)\in M_{1} \bigcup M_{2}}{\sum}  
 a_{-m,-n} U^{-m} V^{-n}
$$
$$
=1+ \underset{(m,n)\in M_{1} \bigcup M_{2}}{\sum}  
 a_{m,n} U^{m} V^{n}+\underset{(m,n)\in M_{1} \bigcup M_{2}}{\sum}  
 \overline{a_{m,n}} \mathit{e}^{-2\pi imn \theta } U^{-m} V^{-n}
$$
$$
=1+ \underset{(m,n)\in M_{1} \bigcup M_{2}}{\sum}  
 a_{m,n} U^{m} V^{n}+\underset{(m,n)\in M_{1} \bigcup M_{2}}{\sum}  
 \overline{a_{m,n}} \mathit{e}^{-2\pi imn \theta } \mathit{e}^{2\pi imn \theta }V^{-n} U^{-m} 
$$
$$
=1+ \underset{(m,n)\in M_{1} \bigcup M_{2}}{\sum}  
 a_{m,n} U^{m} V^{n}+\underset{(m,n)\in M_{1} \bigcup M_{2}}{\sum}  
 \overline{a_{m,n}} V^{-n} U^{-m}.
$$

By the same reasoning we can show 
$$
P_{r}(a)= 1+\underset{(m,n)\in M}{\sum}  
 a_{m,n} r^{(\vert m \vert + \vert n \vert)} U^{m} V^{n}
$$

$$
=1+ \underset{(m,n)\in M_{1} \bigcup M_{2}}{\sum}  
 a_{m,n} r^{(\vert m \vert + \vert n \vert)} U^{m} V^{n}$$
 $$+\underset{(m,n)\in M_{1} \bigcup M_{2}}{\sum}  
 \overline{a_{m,n}} r^{(\vert m \vert + \vert n \vert)} V^{-n} U^{-m}.
$$
Since $a$ is strictly positive, using (\ref{a=}), we see that 

$$
\phi(a)= 1+ \underset{(m,n)\in M_{1} \bigcup M_{2}}{\sum}  
 a_{m,n} \phi (U^{m} V^{n})$$
 $$+\underset{(m,n)\in M_{1} \bigcup M_{2}}{\sum}  
 \overline{a_{m,n}} \phi (V^{-n} U^{-m})>0.
$$
Let $h_{m,n}=a_{m,n} \phi (U^{m} V^{n})$. Then regarding the fact that 
$$\phi (U^{m} V^{n})= \overline{\phi (V^{-n} U^{-m})},$$
we have 
\begin{equation} \label{aispos}
\phi(a)= 1+ \underset{(m,n)\in M_{1} \bigcup M_{2}}{\sum}  
 (h_{m,n}+\overline{h_{m,n}})>0.
\end{equation}
On the other hand, 
$$
\phi (P_{r}(a))=1+ \underset{(m,n)\in M_{1} \bigcup M_{2}}{\sum}  
 a_{m,n} r^{(\vert m \vert + \vert n \vert)} \phi (U^{m} V^{n})
$$
$$
+\underset{(m,n)\in M_{1} \bigcup M_{2}}{\sum}  
 \overline{a_{m,n}}r^{(\vert m \vert + \vert n \vert)} \phi (V^{-n} U^{-m})=t_{0}\leq0.
$$
So
\begin{equation} \label{pa=}
\phi (P_{r}(a))=1+ \underset{(m,n)\in M_{1} \bigcup M_{2}}{\sum}  
 r^{(\vert m \vert + \vert n \vert)}(h_{m,n}+\overline{h_{m,n}})\leq0.
\end{equation}
Then let 
$$r^{(\vert m_{0} \vert + \vert n_{0} \vert)}=\text{Min}\left \{r^{(\vert m \vert + \vert n \vert)} : \, (m,n)\in M \right \} $$
and note that $0\leq r^{(\vert m_{0} \vert + \vert n_{0} \vert)}\leq 1$. Now we have two cases. Either $$  -1<\underset{(m,n)\in M_{1} \bigcup M_{2}}{\sum}  
 (h_{m,n}+\overline{h_{m,n}})\leq0,$$ 
or
\begin{equation} \label{cas2}
\underset{(m,n)\in M_{1} \bigcup M_{2}}{\sum}  
 (h_{m,n}+\overline{h_{m,n}})>0.
\end{equation}
In the first case, since 

$$
\underset{(m,n)\in M_{1} \bigcup M_{2}}{\sum}(h_{m,n}+\overline{h_{m,n}}) \leq  r^{(\vert m_{0} \vert + \vert n_{0} \vert)} \underset{(m,n)\in M_{1} \bigcup M_{2}}{\sum}(h_{m,n}+\overline{h_{m,n}}),
$$ 
we have
$$
-1<\underset{(m,n)\in M_{1} \bigcup M_{2}}{\sum}  
 r^{(\vert m_{0} \vert + \vert n_{0} \vert)}(h_{m,n}+\overline{h_{m,n}})$$
 $$ \leq\underset{(m,n)\in M_{1} \bigcup M_{2}}{\sum} r^{(\vert m \vert + \vert n \vert)}  (h_{m,n}+\overline{h_{m,n}}).
$$
So $$\underset{(m,n)\in M_{1} \bigcup M_{2}}{\sum}  
 r^{(\vert m \vert + \vert n \vert)}(h_{m,n}+\overline{h_{m,n}})>-1,$$ which contradicts \ref{pa=}. In the second case again we have
 $$
 \underset{(m,n)\in M_{1} \bigcup M_{2}}{\sum}  
 r^{(\vert m_{0} \vert + \vert n_{0} \vert)}(h_{m,n}+\overline{h_{m,n}})$$
 $$\leq \underset{(m,n)\in M_{1} \bigcup M_{2}}{\sum} r^{(\vert m \vert + \vert n \vert)}  (h_{m,n}+\overline{h_{m,n}})\leq-1,
 $$
which means $$
\underset{(m,n)\in M_{1} \bigcup M_{2}}{\sum}  
 r^{(\vert m_{0} \vert + \vert n_{0} \vert)}(h_{m,n}+\overline{h_{m,n}})
$$
is strictly negative. But this contradicts \ref{cas2}, for 
$$
\underset{(m,n)\in M_{1} \bigcup M_{2}}{\sum}  
 r^{(\vert m_{0} \vert + \vert n_{0} \vert)}(h_{m,n}+\overline{h_{m,n}})
$$
and

$$
\underset{(m,n)\in M_{1} \bigcup M_{2}}{\sum}  
 (h_{m,n}+\overline{h_{m,n}})
$$
have the same signs.

Then we show there exist $B_{1}, B_{2}>0$ such that for $0\leq r\leq 1$, $\sigma (P_{r}(a)) \subseteq [B_{1}, B_{2}]$, where $\sigma (P_{r}(a))$ is the spectrum of $P_{r}(a)$. Since for $0\leq r\leq 1$,
$$
\Vert P_{r}(a) \Vert = \Vert 1+\underset{(m,n)\in M}{\sum}  
 a_{m,n} r^{(\vert m \vert + \vert n \vert)} U^{m} V^{n} \Vert 
$$
$$
\leq 1+\underset{(m,n)\in M}{\sum}  
  \vert a_{m,n} \vert r^{(\vert m \vert + \vert n \vert)}  \Vert U^{m} V^{n} \Vert
  \leq 1+\underset{(m,n)\in M}{\sum}  
  \vert a_{m,n} \vert,
$$
and we know the spectral rasius of $ P_{r}(a)$ is less than $
\Vert P_{r}(a)\Vert $, it suffices to put $$B_{2}=1+\underset{(m,n)\in M}{\sum}  
  \vert a_{m,n} \vert.$$

Now suppose there is no such $B_{1}$. So for each $n >0$ there exists $r_{n} \in [0,1]$, and $\lambda _{n} \in (0, \dfrac{1}{n})$, such that 
$\lambda _{n} \in \sigma (P_{r_{n}}(a))$. Obviously $\underset{n \to \infty}{\lim} \lambda _{n}=0 $.
 Since $\left \{ r_{n} \right \}_{n=1}^{\infty}$ 
is a bounded sequence, it has a convergent subsequence. 
For simplicity we wil call that sequence again 
$\left \{ r_{n} \right \}_{n=1}^{\infty}$. Let $\underset{n \to \infty}{\lim} r_{n}= r_{0}$. 
Then $ \underset{n \to \infty}{\lim}P_{r_{n}}(a)=P_{r_{0}}(a)$. 
Let $\text{Inv}(A_{\theta})$ be the set of invertible elements in $A_{\theta}$. It is an open set, hence its complement is closed. Then for $n>0$, since $\lambda _{n} \in \sigma (P_{r_{n}}(a))$, $$P_{r_{n}}(a)-\lambda _{n} 1\notin \text{Inv}(A_{\theta}).$$
Then $$
\underset{n \to \infty}{\lim}P_{r_{n}}(a)-1\lambda _{n}
=P_{r_{0}}(a)\notin \text{Inv}(A_{\theta}),
$$
which means $0 \in \sigma (P_{r_{0}}(a))$. But this is a contradiction, for we have shown for $0\leq r \leq 1$, $P_{r}(a)$ is strictly positive.

Now we pick a neighborhood of $[B_{1}, B_{2}]$ away from the $y$-axis. Let 
$$N=\left \{ x+iy : \, \frac{2B_{1}}{3}\leq x \leq B_{2}+1, -1\leq y\leq 1  \right \}.$$ 
Clearly for $0\leq r \leq 1$, $\sigma (P_{r}(a)) \subseteq N$. So by Proposition \ref{Rudin}, for $0\leq r \leq 1$, there exists $\delta _{r}$ such that for $y \in A_{\theta}$, 
$\Vert y-P_{r}(a) \Vert < \delta _{r}$ implies  $\sigma (y) \subseteq N$.
Since 
$$ P(a): \mathbb{C} \longrightarrow A_{\theta}, \quad r \mapsto P_{r}(a)$$
 is a continuous map, for $\delta _{r}$ there exists $\gamma_{r}>0$, such that for $r' \in \mathbb{C}$, $\vert r'-r \vert \leq \gamma_{r} $ implies $\Vert P_{r'}(a)-P_{r}(a) \Vert < \delta _{r} $. So if $r' \in B_{\gamma _{r}}(r)$, then  $\sigma (P_{r'}(a)) \subseteq N$ where $B_{\gamma _{r}}(r)$ is the 2-dimensional open ball centred at $r$ with radius $\gamma _{r}$. Now let $$W=\underset{0\leq r\leq 1}{\bigcup} B_{\gamma _{r}}(r).$$
Obviously $W$ is a complex open neighborhood of the interval $[0,1]$ and the way that we have constructed $W$ implies if $r \in W$, then $\sigma (P_{r}(a) \subseteq N$. Since $N$ is in the right half plane, using the standard branch of the logarithm, for $r \in W$, we can define $\log P_{r}(a)$. 
\end{proof}
   \section{ The Main Result}
   \label{prelim}
In this section we will first state our conjecture about a
logarithmic Sobolev inequality on the noncommutative 2-torus and then we will prove it for certain elements. 

\begin{conj} \label{conjec}
 Let $a= \underset{(m,n)\in \mathbb{Z}^{2}}{\sum}  
 a_{m,n} U^{m} V^{n} $ be  
in $A_{\theta}^{\infty}$ and assume  $a>0$.  Then  
\begin{equation} \label{lse}
\tau(a^{2} \log a)\leqslant \underset{(m,n)\in \mathbb{Z}^{2}}{\sum} (\vert m\vert + \vert n\vert) \vert a_{m,n} \vert ^{2} + \Vert a \Vert _{2} ^{2} \log \Vert a \Vert _{2},
\end{equation}
 which is the same as
$$\tau(a^{2} \log a)\leqslant \underset{(m,n)\in \mathbb{Z}^{2}}{\sum} (\vert m\vert + \vert n\vert) \vert a_{m,n} \vert ^{2} + \tau (a^{2})\log ( \tau (a))^{1/2}. $$
\end{conj} 
Our main goal was of course to prove the conjecture in general using Weissler's method \cite{W}, however, because of noncommutativity, in the last step  we encountered a technical problem. So we decided to restrict ourselves to a class of elements. Now we will prove the conjecture for the case $m=sn$ for some $s$ and later on in section 4 we will give more details of what we heve set up for the general case and explain what the problem is in this setting. 
\begin{theorem} \label{man} Let $a= \underset{n\in \mathbb{Z}}{\sum}  
 a_{n} U^{n} V^{sn} $ be  
in $A_{\theta}^{\infty}$ where $s \in \mathbb{Z} \setminus \left \{ 0 \right \}$, such that $a>0$. Then  
\begin{equation} \label{lse}
\tau(a^{2} \log a)\leqslant \underset{n\in \mathbb{Z}}{\sum} (1+ \vert s \vert ) \vert n\vert \vert a_{n} \vert ^{2} + \Vert a \Vert _{2} ^{2} \log \Vert a \Vert _{2}
\end{equation}
 which is the same as
$$\tau(a^{2} \log a)\leqslant \underset{n\in \mathbb{Z}}{\sum}(1+ \vert s \vert ) \vert n\vert \vert a_{n} \vert ^{2}+ \tau (a^{2})\log ( \tau (a))^{\dfrac{1}{2}}. $$
\end{theorem}
   
\begin{proof} 
First suppose $\tau (a)=1$ \textit{i.e}. $a_{0}=1$, and suppose that at most finitely many number of $a_{n}$ are nonzero. Put $x=a-1= \underset{\underset{n\neq 0}{n\in \mathbb{Z}}}{\sum} a_{n} U^{n} V^{sn}$. Using the fact that
$\Vert a \Vert _{2}^{2}=1+\Vert x \Vert _{2} ^{2}$, it can be shown that $$\Vert a \Vert _{2} ^{2} \log \Vert a \Vert _{2} \geqslant \dfrac{1}{2} \Vert x \Vert _{2}^{2}.$$
We shall prove the theorem by proving an stronger inequality, namely

\begin{equation} \label{als}
 0\leqslant\underset{n\in \mathbb{Z}}{\sum}(1+ \vert s \vert ) \vert n\vert \vert a_{n} \vert ^{2}+ \dfrac{1}{2} \Vert x \Vert _{2}^{2}- \tau(a^{2} \log a).
\end{equation}
For complex number $r$ we define
$$x_{r}=\underset{\underset{n\neq 0}{n \in \mathbb{Z}}}{\sum} a_{n} r^{ (1+ \vert s \vert )\vert n\vert} U^{n} V^{sn}$$
and $P_{r}(a)=1+x_{r}$. By Proposition \ref{comp}, for $r$ in some complex neighborhood of the interval $[0,1]$,
we can define $\log P_{r}(a)$. 
 
Let
$$G(r)=\underset{n\in \mathbb{Z}}{\sum}  (1+ \vert s \vert )r^{2(1+ \vert s \vert )\vert n\vert } \vert n\vert \vert a_{n} \vert^{2} 
$$
$$+  \dfrac{1}{2} \Vert x_{r} \Vert _{2}^{2}-  \tau((P_{r}(a))^{2} \log P_{r}(a)).$$
Therefore, to prove (\ref{als}) it suffices to show $G(1)\geqslant 0$. It can be shown that $G(r)$ is analytic in a complex neighborhood of $[0,1]$. So to prove $G(1)\geqslant 0$, using Lemma \ref{ana}, we shall show that all the coefficients of the expansion of $G(r)$ around $r=0$ are nonnegative.
 
First note that for $r$ with small enough $\vert r \vert$ we have $\Vert x_{r} \Vert_{2} < 1$ (the sum is a finite sum). So 
$$(P_{r}(a))^{2} \log P_{r}(a)=
(1+x_{r})^{2} \log (1+x_{r}) $$
$$=(1+2x_{r}+ x_{r}^{2})
(1-\dfrac{1}{2}x_{r}^{2}+\dfrac{1}{3}x_{r}^{3}-\dfrac{1}{4}x_{r}^{4}+ \cdots ) $$
$$= x_{r}+ \dfrac{3}{2} x_{r}^{2} 
+ 2\sum_{k=3}^{\infty} (-1)^{k-1} x_{r}^{k} \dfrac{(k-3)!}{k!}.$$
So 
\begin{equation} \label{gr}
G(r)=\underset{n\in \mathbb{Z}}{\sum} (1+ \vert s \vert ) r^{2 (1+ \vert s \vert )\vert n\vert} \vert n\vert\vert a_{n} \vert^{2} 
+  \dfrac{1}{2} \Vert x_{r} \Vert _{2}^{2}
\end{equation}

$$-  \tau(x_{r})- \dfrac{3}{2}\tau(x_{r}^{2})+ 2\sum_{k=3} ^{\infty} (-1)^{k}\dfrac{(k-3)!}{k!} \tau (x_{r}^{k}). $$
Using the facts that $\tau(x_{r})=0$ and $$\tau (x_{r}^{2})= \Vert x_{r}\Vert^{2}_{2}= \underset{\underset{n\neq 0}
{n\in \mathbb{Z}}}{\sum}  r^{2(1+ \vert s \vert )\vert n\vert} \vert a_{n}\vert ^{2},$$
combined with (\ref{gr}), we get
\begin{equation} \label{gro}
G(r)= 2\underset{\underset{n\geqslant0}
{n\in \mathbb{Z}}}{\sum}  ( (1+ \vert s \vert ) n -1)r^{2(1+ \vert s \vert ) n} \vert a_{n}\vert ^{2} +2\sum_{k=3}^{\infty} g_{k}(r),
\end{equation}
where $g_{k}(r)=(-1)^{k}\dfrac{(k-3)!}{k!} \tau (x_{r}^{k})$. Now we try to find the Taylor expansion of $\tau(x_{r}^{k})$.

First we need to fix some notations. Let$$ M=\lbrace n \in \mathbb{Z}: \, n\neq 0,a_{n} \neq0 \rbrace,$$
$$M_{1}=\lbrace n \in M : \, n>0 \rbrace.$$  For a function $P:M \longrightarrow \mathbb{Z}_{0}^{+}$, we put 
$$
M_{P}= \left \{ n \in M : \, P(n) \neq  0 \right \}.
$$
So $(M_{P},P)$ is a multiset. Indeed, the multiplicity of $n$ is $P(n)$.
Moreover, let $\mathcal{S} (M_{P})$ be the set of all permutations of the multiset $(M_{P},P)$.
Let $I_{k}$ be the set of all functions $P:M \longrightarrow \mathbb{Z}_{0}^{+}$
such that $$\underset{n \in M}{\sum} P(n)=k,$$
and $I_{k,0}$ be the set of all functions in $I_{k}$ such that 
$$
\underset{n \in M}{\sum} P(n)n=0.
$$
 For $P:M \longrightarrow \mathbb{Z}_{0}^{+}$, we also define  $$Q_{P}:M_{1} \longrightarrow \mathbb{Z}_{0}^{+}$$ by
$Q_{P}(n)=P(-n)$.
 
 Using the  multinomial expansion of $x_{r}$ we have
$$x_{r}^{k}=\sum_{P \in I_{k}} \left( \underset{n\in M}{\prod} \left( a_{n}r^{(1+ \vert s \vert ) \vert n \vert} \right)^{P(n)}\right)
\left(\underset{\sigma \in \mathcal{S}(M_{P})} {\sum}
  \prod_{i=1}^{k} U^{\sigma(n_{i}^{P})} V^{s\sigma(n_{i}^{P})}
  \right)$$
where $n_{i}^{P}$, for $i=1,2, \cdots k$, is a labeling of elements of $M_{P}$ when $P \in I_{k}$. 
Then $$ \tau (x_{r}^{k})=\sum_{P \in I_{k,0}} \left( \underset{n\in M}{\prod} \left( a_{n}r^{(1+ \vert s \vert )\vert n \vert} \right)^{P(n)}\right)
\tau \left(\underset{\sigma \in \mathcal{S}(M_{P})} {\sum}
  \prod_{i=1}^{k} U^{\sigma(n_{i}^{P})} V^{s\sigma(n_{i}^{P})}
  \right).$$
 
 So we have
 $$ \tau (x_{r}^{k})=\sum_{P \in I_{k,0}} \underset{n\in M}{\prod} \left( a_{n}r^{(1+ \vert s \vert ) n } \right)^{P(n)}
 \underset{-n\in M_{1}}{\prod} \left( a_{n}r^{ -(1+ \vert s \vert )  n} \right)^{P(n)}\times 
$$
$$
\tau \left(\underset{\sigma \in \mathcal{S}(M_{P})} {\sum}
  \prod_{i=1}^{k} U^{\sigma(n_{i}^{P})} V^{s\sigma(n_{i}^{P})}
  \right).
$$
And hence
$$ \tau (x_{r}^{k})=\sum_{P \in I_{k,0}}  \underset{n\in M_{1}}{\prod} \left( a_{n}r^{ (1+ \vert s \vert )  n} \right)^{P(n)}
 \underset{n\in M_{1}}{\prod} \left( a_{-n}r^{ (1+ \vert s \vert )n} \right)^{Q_{P}(n)}
$$

$$\tau \left(\underset{\sigma \in \mathcal{S}(M_{P})} {\sum}
  \prod_{i=1}^{k} U^{\sigma(n_{i}^{P})} V^{s\sigma(n_{i}^{P})}
  \right).
$$

Then since $a$ is self-adjoint, using (\ref{sa}) we have
\begin{equation} \label{mmm}
\tau (x_{r}^{k})=\sum_{(P,Q) \in H_{k}}  \underset{n\in M_{1}}{\prod} \left( a_{n}r^{ (1+ \vert s \vert )n} \right)^{P(n)}
 \underset{n\in M_{1}}{\prod} \left( \overline{a_{n}}r^{(1+ \vert s \vert )n} \right)^{Q(n)}
\end{equation}

$$ e^{(-2 \pi is\theta  \underset{n\in M_{1}}{\sum} Q(n)n^{2})} 
\tau \left(\underset{\sigma \in \mathcal{S}(M_{P,Q})} {\sum}
  \prod_{i=1}^{k} U^{\sigma(n_{i}^{P,Q})} V^{s\sigma(n_{i}^{P,Q})}
  \right),
$$
where $H_{k}$ is the set of all pairs $(P,Q)$ such that $$P:M_{1} \longrightarrow \mathbb{Z}_{0}^{+},$$  $$Q:M_{1} \longrightarrow \mathbb{Z}_{0}^{+},$$ 
\begin{equation}\label{8}
\underset{n \in M_{1}}{\sum} P(n)n =\underset{n \in M_{1} }{\sum} Q(n)n,
\end{equation}

\begin{equation}\label{10}
\underset{n \in M_{1}}{\sum} P(n) + \underset{n \in M_{1}}{\sum} Q(n)=k.
\end{equation}
Also
$(M_{P,Q}, [P,Q])$ is a multiset defined by $$
M_{P,Q}=M_{P}^{+} \cup M_{Q}^{-},
$$
where
$$
M_{P}^{+}=\left \{ n \in M_{1}  : \, P(n) \neq  0 \right \},
$$

$$
M_{Q}^{-}=\left \{ n \in M : \, -n \in M_{1},  Q(-n) \neq  0 \right \},
$$ 
and
$$
[P,Q]:M_{P,Q} \longrightarrow \mathbb{Z}_{0}^{+},
$$ 
is defined by
$$[P,Q](n)=\left\{\begin{matrix}
P(n) &  n \in M_{P}^{+} \\
Q(-n) & n \in M_{Q}^{-}
\end{matrix}.\right.
$$
Also, $n_{i}^{P,Q}$ for $i=1,2, \cdots,k$ is a labeling for elements of $M_{P,Q}$.

So regarding (\ref{mmm}), we see
$$
\tau (x_{r}^{k})=\sum_{(P,Q) \in H_{k}}  \underset{n\in M_{1}}{\prod} \left( a_{n}r^{(1+ \vert s \vert )  n} \right)^{P(n)}
 \underset{n\in M_{1}}{\prod} \left( \overline{a_{n}}r^{ (1+ \vert s \vert )n} \right)^{Q(n)}
$$
$$ e^{(-2 \pi is\theta  \underset{n\in M_{1}}{\sum} Q(n)n^{2})}
\underset{\sigma \in \mathcal{S}(M_{P,Q})} {\sum} \tau \left( 
  \prod_{i=1}^{k}U^{\sigma(n_{i}^{P,Q})} V^{s\sigma(n_{i}^{P,Q})} \right).
$$
 Now we calculate $\tau \left( 
  \prod_{i=1}^{k}U^{\sigma(n_{i}^{P,Q})} V^{s \sigma(n_{i}^{P,Q})} \right)$, for $ \sigma \in \mathcal{S}(M_{P,Q})$, the set of permutations of the multiset $M_{P,Q}$.  For simplicity we drop the superscript $P,Q$. Using (\ref{8}), we have 
  \begin{equation} \label{sig0}
   \underset{i=1}  {\overset{k}\sum} \sigma(n_{i})=0
  \end{equation}
Hence

$$\tau \left( 
  \prod_{i=1}^{k}U^{\sigma(n_{i})} V^{s\sigma(n_{i})} \right)= 
  \tau \left( e^{2 \pi i \theta B_{\sigma} } U^{ \underset{i=1}  {\overset{k}\sum} \sigma(n_{i})} V^{ s\underset{i=1}  {\overset{k}\sum} \sigma(n_{i})} \right)=e^{2 \pi i \theta B_{\sigma} }, $$
where
 for $ \sigma \in \mathcal{S}(M_{P,Q})$, 
 $$B_{\sigma}=\dfrac{s}{2} \underset{n \in M_{1}}{\sum} (P(n) + Q(n))n^{2}.$$

Infact, we know that
  $$
  B_{\sigma}= -s\sigma(n_{2})\sigma(n_{1}) $$
  $$ 
  -s\sigma(n_{3})\left[\sigma(n_{1})+\sigma(n_{2})\right]$$
  $$
  -s\sigma(n_{4})\left[\sigma(n_{1})+\sigma(n_{2})+\sigma(n_{3})\right]-
  \cdots $$
  $$
  - s\sigma(n_{k-1})\left[\sigma(n_{1})+\sigma(n_{2})+ \cdots + \sigma(n_{k-2})\right].
  $$
  $$
 - s\sigma(n_{k})\left[\sigma(n_{1})+\sigma(n_{2})+ \cdots + \sigma(n_{k-1})\right].
  $$
We also define 
$$
A_{\sigma}=s\sigma(n_{1})\left[\sigma(n_{1})+\sigma(n_{2})+ \cdots + \sigma(n_{k})\right]$$
$$+s\sigma(n_{2})\left[\sigma(n_{2})+\sigma(n_{3})+ \cdots + \sigma(n_{k})\right]
$$
$$+s\sigma(n_{3})\left[\sigma(n_{3})+\sigma(n_{4})+ \cdots + \sigma(n_{k})\right] $$
$$+s\sigma(n_{4})\left[\sigma(n_{4})+\sigma(n_{5})+ \cdots + \sigma(n_{k})\right]+ \cdots$$
$$+s\sigma(n_{k-1})\left[\sigma(n_{k-1})+ \sigma(n_{k})\right]$$
$$+s\sigma(n_{k})\sigma(n_{k}).$$
Using (\ref{sig0}), we get $B_{\sigma}-A_{\sigma}=0$. So $B_{\sigma}=\dfrac{1}{2}(B_{\sigma}+A_{\sigma})$. On the other hand,
$$
B_{\sigma}+A_{\sigma}= \underset{i=1}  {\overset{k}\sum} s(\sigma(n_{i}))^{2}=\underset{i=1}  {\overset{k}\sum} s(n_{i})^{2}$$
$$=\sum_{n \in M_{P}^{+}} P(n)sn^{2}+\sum_{n \in M_{Q}^{-}}Q(-n)sn^{2}=\sum_{n \in M_{1}} P(n)sn^{2}+\sum_{-n \in M_{1}} Q(-n)sn^{2}
$$
$$
=\sum_{n \in M_{1}} P(n)sn^{2}+\sum_{n \in M_{1}} Q(n)sn^{2}.
$$
So we have proved 
 $$B_{\sigma}=\dfrac{s}{2} \underset{n \in M_{1}}{\sum} (P(n) + Q(n))n^{2}.$$
Therefore,

$$
\tau (x_{r}^{k})=\sum_{(P,Q) \in H_{k}}  \underset{n\in M_{1}}{\prod} \left( a_{n}r^{ (1+ \vert s \vert ) n} \right)^{P(n)}
 \underset{n \in M_{1}}{\prod} \left( \overline{a_{n}}r^{ (1+ \vert s \vert )n} \right)^{Q(n)}
$$
$$ e^{(-2 \pi is\theta  \underset{n\in M_{1}}{\sum} Q(n)n^{2})}
\underset{\sigma \in \mathcal{S}(M_{P,Q})} {\sum} e^{2 \pi i \theta B_{\sigma} }.
$$

 Since $$\vert \mathcal{S}(M_{P,Q}) \vert =\dfrac{k!}{\underset{n\in M_{1}}{\prod} P(n)!Q(n)!},$$
we see that
$$
\tau (x_{r}^{k})=k!\sum_{(P,Q) \in H}  \underset{n\in M_{1}}{\prod} \left( a_{n}r^{ (1+ \vert s \vert )  n} \right)^{P(n)}
 \underset{n\in M_{1}}{\prod} \left( \overline{a_{n}}r^{ (1+ \vert s \vert )n} \right)^{Q(n)} \times
$$
$$ \dfrac{1}{\underset{n\in M_{1}}{\prod} P(n)!Q(n)!}  e^{-2 \pi is\theta  \underset{(m,n)\in M_{1}}{\sum} Q(n)n^{2}}   
 e^{\pi is \theta \left(\underset{n \in M_{1}}{\sum} (P(n) + Q(n))n^{2}\right)}
$$
$$
=k!\sum_{(P,Q)\in H}  \underset{n\in M_{1}}{\prod} \left( a_{n}r^{ (1+ \vert s \vert )  n} \right)^{P(n)}
 \underset{n\in M_{1}}{\prod} \left( \overline{a_{n}}r^{ (1+ \vert s \vert )n} \right)^{Q(n)}\times
$$
$$
 \dfrac{1}{\underset{n\in M_{1}}{\prod} P(n)!Q(n)!}e^{\pi is \theta  \underset{n \in M_{1}}{\sum} (P(n) - Q(n))n^{2}}$$
$$
=k!\sum_{(P,Q) \in H} r^{\left(\underset{n \in M_{1}}{\sum} (1+ \vert s \vert )n ( P(n)+Q(n)) \right)}\times  $$
$$e^{\pi i  s\theta  \underset{n \in M_{1}}{\sum} P(n)n^{2}} \underset{n\in M_{1}}  {\prod} \dfrac{\left( a_{n} \right) ^{P(n)}}{P(n)!}
 e^{-\pi is \theta  \underset{n \in M_{1}}{\sum} Q(n)n^{2}} \underset{n\in M_{1}}{\prod} \dfrac{\left( \overline{a_{n}} \right)^{Q(n)}}{Q(n)!}.
$$

Now for a function $P:M \longrightarrow \mathbb{Z}_{0}^{+},$ define 
\begin{equation} \label{C(P)}
D(p)= e^{\pi i s\theta\underset{n \in M_{1}}{\sum} P(n)n^{2}} \underset{n\in M_{1}}{\prod} \dfrac{\left( -a_{n} \right)^{P(n)}}{P(n)!}.
\end{equation} 
Then we have
$$
\tau (x_{r}^{k})=k!\sum_{(P,Q) \in H} (-1)^{  (\underset{n \in M_{1}}{\sum} P(n) +  Q(n))} r^{(\underset{n \in M_{1}}{\sum} (1+ \vert s \vert )n\left( P(n)+Q(n) \right)) }D(P)\overline{D(Q)} 
$$

So
\begin{equation} \label{aaa}
\tau (x_{r}^{k})=(-1)^{k}k!\sum_{l=2}^{\infty} r^{2(1+ \vert s \vert )l}
\left( 
 \sum_{(P,Q) \in G_{l}} D(P)\overline{D(Q)} \right)
\end{equation}

 where $G_{l}$ the set of all pairs $(P,Q)$ such that $$P:M_{1} \longrightarrow \mathbb{Z}_{0}^{+},$$  $$Q:M_{1} \longrightarrow \mathbb{Z}_{0}^{+},$$ and
$$
\underset{n \in M_{1}}{\sum} P(n) + \underset{n \in M_{1}}{\sum} Q(n)=k,
$$
$$
\underset{n \in M_{1}}{\sum} P(n)n = \underset{n \in M_{1} }{\sum} Q(n)n =l,
$$

One should note that in (\ref{aaa}), $l$ starts from $2$. Here we shall show why that is the case:
\begin{equation}
\underset{n \in M_{1}}{\sum} P(n)\leqslant \underset{n \in M_{1}}{\sum} P(n)n.
\end{equation}
Similarly we have 
\begin{equation} \label{ccc}
\underset{n \in M_{1}}{\sum} Q(n)\leqslant \underset{n \in M_{1}}{\sum} Q(n)n,
\end{equation}

So $$k=\underset{n \in M_{1}}{\sum} P(n) + \underset{n \in M_{1}}{\sum} Q(n)$$
$$\leqslant\underset{n \in M_{1}}{\sum} P(n)n + \underset{n \in M_{1}}{\sum} Q(n)n =2l$$
So for a fixed $k$, $\dfrac{k}{2} \leqslant l$ and since $k$ is at least $3$, $l \geqslant 2$. 

Now for $l$ and $t\geqslant 1$ define 
$$C(t,l):=\underset {P \in H_{t,l}}{\sum} D(P),$$
where $H_{t,l}$ is the set of all functions $P:M_{1} \longrightarrow \mathbb{Z}_{0}^{+}$
such that
\begin{equation} \label{t}
\underset{n \in M_{1}}{\sum} P(n)=t,
\end{equation}
and
\begin{equation} \label{l}
\underset{n \in M_{1}}{\sum} P(n)n=l.
\end{equation}
When there is no such a $P$ then the sum is taken to be $0$. For instance 
if $t>l$ then there is no such a $P$, for 
$$t=\underset{n \in M_{1}}{\sum} P(n) \leqslant
 \underset{n \in M_{1}}{\sum} P(n)n=l.$$
Then we have 
$$ \underset {(P,Q) \in G_{l}}{\sum} D(P)\overline{D(Q)} 
 =\sum_{t=1}^{k-1}\underset {\underset{\mathcal{P} \in H_{t,l}}{\mathcal{Q} \in H_{k-t,l}} }{\sum} D(\mathcal{P})\overline{D(\mathcal{Q})}
$$
$$= \sum_{t=1}^{k-1} \left( \underset {\mathcal{P} \in H_{t,l}}{\sum} D(\mathcal{P})\right)
\left(  \underset {\mathcal{Q} \in H_{k-t,l}}{\sum} \overline{D(\mathcal{Q})}\right)=\sum_{t=1}^{k-1} C(t,l)\overline{C(k-t,l)}.$$
Now using this in (\ref{aaa}), we get
$$\tau (x_{r}^{k})=(-1)^{k}k!\sum_{l=2}^{\infty} r^{2(1+ \vert s \vert )l}\sum_{t=1}^{k-1} C(t,l)\overline{C(k-t,l)}$$ and this implies
$$\sum_{k=3}^{N} g_{k}(r)=
\sum_{k=3}^{N}(k-3)!\sum_{l=2}^{\infty} r^{2(1+ \vert s \vert )l}\sum_{t=1}^{k-1} C(t,l)\overline{C(k-t,l)}$$
$$
=\sum_{l=2}^{\infty} r^{2(1+ \vert s \vert )l}\sum_{k=3}^{N}(k-3)!
\sum_{t=1}^{k-1} C(t,l)\overline{C(k-t,l)}
$$
$$
=\sum_{l=2}^{\infty} r^{2(1+ \vert s \vert )l}\underset{3\leq i+j\leq N}{\sum_{i=1}^{l}\sum_{j=1}^{l}}
(i+j-3)! C(i,l)\overline{C(j,l)}.
$$
Therefore, for $N\geqslant 2l$ the coefficient of $r^{2(1+ \vert s \vert )l}$ 
in $\sum_{k=3}^{N} g_{k}(r)$ is
\begin{equation} \label{Cogk}
\underset{i+j\geq 3}{\sum_{i=1}^{l}\sum_{j=1}^{l}}(i+j-3)! C(i,l)\overline{C(j,l)},
\end{equation}
and this is also the coefficient of of $r^{2(1+ \vert s \vert )l}$ in $\sum_{k=3}^{\infty} g_{k}(r)$.

Now we show for $l\geq0$,  $C(1,l)=-a_{l} e^{\pi isl^{2}\theta}$. Infact this is true, for if $l\notin M$, then  $a_{l}=0$ and also $H_{1,l}=\varnothing $ which implies $C(1,l)=0$. If $l\in M$, then
$$P(n)=
\left\{\begin{matrix}
1 & n=l\\ 
0 & \text{otherwise}
\end{matrix}\right.
$$
 is the only element in $H_{1,l}$. So using (\ref{C(P)}), we have
\begin{equation} \label{c1ls}
 C(1,l)=\underset {P \in H_{1,l}}{\sum} D(P)=-a_{l} e^{\pi isl^{2}\theta}.
\end{equation} 

Recall that $$G(r)= 2\underset{\underset{n\geqslant0}
{n\in \mathbb{Z}}}{\sum}  ( (1+ \vert s \vert )n -1)r^{2(1+ \vert s \vert )n} \vert a_{n}\vert ^{2} +2\sum_{k=3}^{\infty} g_{k}(r).$$ 
Therefore the coefficient of $r^{2(1+ \vert s \vert )l}$, ($l\geqslant2$) in $G(r)$ is
$$
2((1+ \vert s \vert )l-1)\vert a_{l}\vert ^{2}+2 \underset{i+j\geq 3}{\sum_{i=1}^{l}\sum_{j=1}^{l}}(i+j-3)! C(i,l)\overline{C(j,l)}.
$$
Using (\ref{c1ls}), this is equal to
$$
2((1+ \vert s \vert )l-1)C(1,l)\overline{C(1,l)}+2 \underset{i+j\geq 3}{\sum_{i=1}^{l}\sum_{j=1}^{l}}(i+j-3)! C(i,l)\overline{C(j,l)}
$$
which can be written as 
\begin{equation} \label{fco}
2\sum_{i=1}^{l}\sum_{j=1}^{l}A_{l}(i,j)C(i,l)\overline{C(j,l)}
\end{equation}
where the matrix $A_{l}$ defined by $$
A_{l}(i,j)=\left\{\begin{matrix}
(1+ \vert s \vert )l-1 & i=j=1 \\ 
(i+j-3)! & i+j\geqslant 3 
\end{matrix}\right..
$$
In \cite{W} it has been shown that for $l\geqslant2$, $A_{l}$   is a positive semi-definite matrix. Hence the coefficient of $r^{2(1+ \vert s \vert )l}$, ($l\geqslant2$) in $G(r)$ is positive. So we have proved ~(\ref{lse}) for a positive  element $a$ with $a_{0}=1$ in which only finitely many coefficients are non-zero.

Homogeneity of (\ref{lse}) implies it should hold for a positive element $a$ (with only finitely many  non-zero coefficients), even if $a_{0} \neq 1$. 

Finally, we shall prove (\ref{lse}) for an arbitrary strictly positive element of the form $ \underset{n\in \mathbb{Z}}{\sum}  
 a_{n} U^{n} V^{sn} $. For  $a= \underset{n\in \mathbb{Z}}{\sum}  
 a_{n} U^{n} V^{sn} $ and $b= \underset{n\in \mathbb{Z}}{\sum}  
 b_{n} U^{n} V^{sn} $ in $A_{\theta}^{\infty}$, we define
 $$  a\ast b= \underset{p\in \mathbb{Z}}{\sum}  
 (a\ast b)_{p} U^{p} V^{sp}$$
 where $(a\ast b)_{p}=a_{p}b_{p}$. We also define $d_{j}$ in $A_{\theta}^{\infty}$ for $j\geqslant0$ by
 $$
 d_{j}=\underset{n\in \mathbb{Z}}{\sum}  
 d_{n}^{j} U^{n} V^{sn},
 $$
where $$
d_{n}^{j}=\left\{\begin{matrix}
1 & \vert n\vert\leq  j\\ 
 0& \text{otherwise}
\end{matrix}\right..
$$
Then we have $$
\Vert (d_{j}\ast a)-a \Vert_{2}^{2}=\underset{n \in \mathbb{Z}}{\sum}\vert d_{n}^{j}a_{n}-a_{n}\vert^{2}=\underset{ \vert n\vert>   j}{\sum}\vert a_{n} \vert^{2}.
$$
So
\begin{equation} \label{lima}
\lim_{j\rightarrow \infty }d_{j}\ast a=a,
\end{equation} 
in $ \Vert .\Vert_{2}$ topology.

Moreover,
$$
\Vert (d_{j}\ast a)-a \Vert=\Vert \underset{n\in \mathbb{Z}}{\sum}(d_{n}^{j}a_{n}-a_{n}) U^{n} V^{sn}\Vert 
$$ 
$$
\leq \underset{ \vert n\vert>   j}{\sum}\vert a_{n} \vert,
$$
which implies
$$
\lim_{j\rightarrow \infty }d_{j}\ast a=a,
$$
in C*-norm topology. Now let $$a= \underset{n\in \mathbb{Z}}{\sum}  
 a_{n} U^{n} V^{sn} $$ be a strictly  positive element in $A_{\theta}^{\infty}$ and $F$ be a  complex open neighborhood of $\sigma (a)$ in the right half plane away from the $y$ axis. (Since $a$ is strictly positive, we can choose such a set.) Therefore, using Proposition (\ref{Rudin}), we see for large enough $j$, $\sigma (d_{j})$ is inside $F$. Hence for large enough $j$, we can define $\log (d_{j}\ast a)$. 

  Since for $j
 \geqslant0$, $d_{j}\ast a$ is an element of  $A_{\theta}^{\infty}$ which has at most finitely many non-zero coefficients, we shall apply (\ref{lse}) to $d_{j}\ast a$ for large enough $j$ and we will get
\begin{equation} \label{conv}
\tau((d_{j}\ast a)^{2} \log (d_{j}\ast a))\leqslant \underset{n\in \mathbb{Z}}{\sum} (1+ \vert s \vert ) \vert n\vert \vert d_{n}^{j} \vert ^{2}\vert a_{n} \vert ^{2} 
\end{equation}
$$
+ \Vert d_{j}\ast a \Vert _{2} ^{2} \log \Vert d_{j}\ast a \Vert _{2}.$$
Let $$h=\underset{n\in \mathbb{Z}}{\sum} (1+ \vert s \vert )^{\dfrac{1}{2}}\vert n\vert^{\dfrac{1}{2}}  a_{n} U^{n} V^{sn}.$$ Then
$$
\Vert d_{j}\ast h \Vert _{2}^{2}=
\underset{n\in \mathbb{Z}}{\sum} (1+ \vert s \vert ) \vert n\vert \vert d_{n}^{j} \vert ^{2}\vert a_{n} \vert ^{2}
$$
and
\begin{equation} \label{limh}
\lim_{j\rightarrow \infty }d_{j}\ast h=h 
\end{equation}  
in $ \Vert .\Vert_{2}$ topology.

Thus,
\begin{equation} \label{limn}
\lim_{j\rightarrow \infty }\underset{n\in \mathbb{Z}}{\sum} (1+ \vert s \vert ) \vert n\vert \vert d_{n}^{j} \vert ^{2}\vert a_{n} \vert ^{2}=\lim_{j\rightarrow \infty } \Vert d_{j}\ast h \Vert _{2}^{2}
\end{equation}
$$
=\lim_{j\rightarrow \infty } \Vert  h \Vert _{2}^{2}=\underset{n\in \mathbb{Z}}{\sum} (1+ \vert s \vert ) \vert n\vert \vert a_{n} \vert ^{2}.
$$
To prove (\ref{lse}), taking the limit of (\ref{conv}) as $j\longrightarrow \infty$, we use (\ref{lima}), (\ref{limn}) and also the continuity of $\tau$ with respect to $\Vert .\Vert_{2}$. Infact, $\tau$ is continuous with respect to $\Vert .\Vert_{2}$, for one can show (See \cite{M} Theorem 3.3.2.) for $a \in A_{\theta}$,
$$
\vert \tau (a) \vert ^{2}\leq \Vert \tau \Vert \tau (a^{*}a).
$$

\end{proof}
\section{ Towards proving the Conjecture}
   \label{prcon}
   In this section, as promised in section 3, we will give the details of what we have done  towards proving the
    Conjecture \ref{conjec} and we will explain what the remaining technical problem is.
    In what follows we will use the assumptions of the Conjecture \ref{conjec}.
     
First suppose $\tau (a)=1,$ \textit{i.e}. $a_{0,0}=1$ and suppose that at most finitely many number of $a_{m,n}$ are nonzero. Put $x=a-1= \underset{\underset{(m,n)\neq (0,0)}{(m,n)\in \mathbb{Z}^{2}}}{\sum} a_{m,n} U^{m} V^{n}$. Using the fact that
$\Vert a \Vert _{2}^{2}=1+\Vert x \Vert _{2} ^{2}$, it can be shown that $$\Vert a \Vert _{2} ^{2} \log \Vert a \Vert _{2} \geqslant \dfrac{1}{2} \Vert x \Vert _{2}^{2}.$$
We are going to prove the conjecture by proving an stronger inequality:

\begin{equation} \label{als}
 0\leqslant\underset{(m,n)\in \mathbb{Z}^{2}}{\sum} (\vert m\vert + \vert n\vert) \vert a_{m,n} \vert ^{2}+ \dfrac{1}{2} \Vert x \Vert _{2}^{2}- \tau(a^{2} \log a).
\end{equation}
For complex number $r$ we define
$$x_{r}=\underset{\underset{(m,n)\neq (0,0)}{(m,n)\in \mathbb{Z}^{2}}}{\sum} a_{m,n} r^{(\vert m\vert +\vert n\vert)} U^{m} V^{n}$$
and $P_{r}(a)=1+x_{r}$. By Proposition \ref{comp}, for $r$ in some complex neighborhood of the interval $[0,1]$,
we can define $\log P_{r}(a)$. 

 Let
$$G(r)=\underset{(m,n)\in \mathbb{Z}^{2}}{\sum}  r^{2(\vert m\vert +\vert n\vert)} (\vert m\vert + \vert n\vert)\vert a_{m,n} \vert^{2} 
$$
$$+  \dfrac{1}{2} \Vert x_{r} \Vert _{2}^{2}-  \tau((P_{r}(a))^{2} \log P_{r}(a)).$$
Therefore, to prove (\ref{als}) it suffices to show that $G(1)\geqslant 0$. It can be shown that $G(r)$ is analytic in a complex neighborhood of $[0,1]$. So to prove $G(1)\geqslant 0$, using lemma \ref{ana}, we need to show that all the coefficients of the expansion of $G(r)$ around $r=0$ are nonnegative.
 
First note that for $r$ with small enough $\vert r \vert$ we have $\Vert x_{r} \Vert_{2} < 1$ (the sum is a finite sum). So 
$$(P_{r}(a))^{2} \log P_{r}(a)=
(1+x_{r})^{2} \log (1+x_{r}) $$
$$=(1+2x_{r}+ x_{r}^{2})
(1-\dfrac{1}{2}x_{r}^{2}+\dfrac{1}{3}x_{r}^{3}-\dfrac{1}{4}x_{r}^{4}+ \cdots ) $$
$$= x_{r}+ \dfrac{3}{2} x_{r}^{2} 
+ 2\sum_{k=3}^{\infty} (-1)^{k-1} x_{r}^{k} \dfrac{(k-3)!}{k!}.$$
So 
\begin{equation} \label{gr}
G(r)=\underset{(m,n)\in \mathbb{Z}^{2}}{\sum}  r^{2(\vert m\vert +\vert n\vert)} (\vert m\vert + \vert n\vert)\vert a_{m,n} \vert^{2} 
+  \dfrac{1}{2} \Vert x_{r} \Vert _{2}^{2}
\end{equation}

$$-  \tau(x_{r})- \dfrac{3}{2}\tau(x_{r}^{2})+ 2\sum_{k=3} ^{\infty} (-1)^{k}\dfrac{(k-3)!}{k!} \tau (x_{r}^{k}). $$
Using the facts that $\tau(x_{r})=0$ and $$\tau (x_{r}^{2})= \Vert x_{r}\Vert^{2}_{2}= \underset{\underset{(m,n)\neq (0,0)}
{(m,n)\in \mathbb{Z}^{2}}}{\sum}  r^{2(\vert m\vert +\vert n\vert)} \vert a_{m,n}\vert ^{2}$$
in (\ref{gr}) we get
\begin{equation} \label{gro}
G(r)= 2\underset{\underset{m\geqslant0,n\geqslant0}
{(m,n)\in \mathbb{Z}^{2}}}{\sum}  ( m + n -1)r^{2( m + n)} \vert a_{m,n}\vert ^{2} +2\sum_{k=3}^{\infty} g_{k}(r)
\end{equation}
where $g_{k}(r)=(-1)^{k}\dfrac{(k-3)!}{k!} \tau (x_{r}^{k})$. Now we try to find the Taylor expansion of $\tau(x_{r}^{k})$.

First we need to fix some notations. Let $$ M=\lbrace (m,n) \in \mathbb{Z}^{2}: \, (m,n)\neq 0,a_{m,n} \neq0 \rbrace,$$
$$M_{1}=\lbrace (m,n) \in M : \, m\geqslant0, n\geqslant0 \rbrace,$$
$$M_{2}=\lbrace (m,n) \in M : \, m>0, n<0 \rbrace.$$For a function $P:M \longrightarrow \mathbb{Z}_{0}^{+}$, we put 
$$
M_{P}= \left \{ (m,n) \in M : \, P(m,n) \neq  0 \right \}.
$$
So $(M_{P},P)$ is a multiset. Indeed, the multiplicity of $(m,n)$ is $P(m,n)$.
Moreover, let $\mathcal{S} (M_{P})$ be the set of all permutations of the multiset $(M_{P},P)$. For $\sigma \in \mathcal{S}(M_{P})$, by $\sigma_{1}(m,n)$ and $\sigma_{2}(m,n)$ we mean the first component of $\sigma(m,n)$ and the second component of that respectively.
Let $I_{k}$ be the set of all functions $P:M \longrightarrow \mathbb{Z}_{0}^{+}$
such that $$\underset{(m,n) \in M}{\sum} P(m,n)=k,$$
and $I_{k,0}$ be the set of all functions in $I_{k}$ such that 
$$
\underset{(m,n) \in M}{\sum} P(m,n)m=\underset{(m,n) \in M}{\sum} P(m,n)n=0.
$$
 For $P:M \longrightarrow \mathbb{Z}_{0}^{+}$, we also define  $$Q_{P}:M_{1}\cup M_{2} \longrightarrow \mathbb{Z}_{0}^{+}$$ by
$Q_{P}(m,n)=P(-m,-n)$.
 
 Using the  multinomial expansion of $x_{r}$ we have
$$x_{r}^{k}=\sum_{P \in I_{k}} \left( \underset{(m,n)\in M}{\prod} \left( a_{m,n}r^{\vert m\vert + \vert n \vert} \right)^{P(m,n)}\right)
$$
$$\times \left(\underset{\sigma \in \mathcal{S}(M_{P})} {\sum}
  \prod_{i=1}^{k} U^{\sigma_{1}(m_{i}^{P},n_{i}^{P})} V^{\sigma_{2}(m_{i}^{P},n_{i}^{P})}
  \right)$$
where $(m_{i}^{P},n_{i}^{P})$, for $i=1,2, \cdots k$, is a labeling of elements of $M_{P}$ when $P \in I_{k}$. 
Then $$ \tau (x_{r}^{k})=\sum_{P \in I_{k,0}} \left( \underset{(m,n)\in M}{\prod} \left( a_{m,n}r^{\vert m\vert + \vert n \vert} \right)^{P(m,n)}\right)
$$
$$\times \tau \left(\underset{\sigma \in \mathcal{S}(M_{P})} {\sum}
  \prod_{i=1}^{k} U^{\sigma_{1}(m_{i}^{P},n_{i}^{P})} V^{\sigma_{2}(m_{i}^{P},n_{i}^{P})}
  \right).$$
 
 So we have
 $$ \tau (x_{r}^{k})=\sum_{P \in I_{k,0}}  \underset{(m,n)\in M_{1}}{\prod} \left( a_{m,n}r^{ m +  n} \right)^{P(m,n)}
 \underset{(-m,-n)\in M_{1}}{\prod} \left( a_{m,n}r^{ -m -  n} \right)^{P(m,n)}
$$
$$ \times \underset{(m,n)\in M_{2}}{\prod} \left( a_{m,n}r^{ m -  n} \right)^{P(m,n)}
 \underset{(-m,-n)\in M_{2}}{\prod} \left( a_{m,n}r^{ n-m} \right)^{P(m,n)}
$$
$$\times\tau \left(\underset{\sigma \in \mathcal{S}(M_{P})} {\sum}
  \prod_{i=1}^{k} U^{\sigma_{1}(m_{i}^{P},n_{i}^{P})} V^{\sigma_{2}(m_{i}^{P},n_{i}^{P})}
  \right)
$$
Then,
$$ \tau (x_{r}^{k})=\sum_{P \in I_{k,0}}  \underset{(m,n)\in M_{1}}{\prod} \left( a_{m,n}r^{ m +  n} \right)^{P(m,n)}
 \underset{(m,n)\in M_{1}}{\prod} \left( a_{-m,-n}r^{ m+n} \right)^{Q_{P}(m,n)}
$$
$$ \times \underset{(m,n)\in M_{2}}{\prod} \left( a_{m,n}r^{ m -  n} \right)^{P(m,n)}
 \underset{(m,n)\in M_{2}}{\prod} \left( a_{-m,-n}r^{ m-n} \right)^{Q_{P}(m,n)}
$$
$$\times\tau \left(\underset{\sigma \in \mathcal{S}(M_{P})} {\sum}
  \prod_{i=1}^{k} U^{\sigma_{1}(m_{i}^{P},n_{i}^{P})} V^{\sigma_{2}(m_{i}^{P},n_{i}^{P})}
  \right).
$$

Then since $a$ is self-adjoint, using (\ref{sa}) we have
\begin{equation} \label{mmm}
\tau (x_{r}^{k})=\sum_{(P,Q) \in H_{k}}  \underset{(m,n)\in M_{1}}{\prod} \left( a_{m,n}r^{ m +  n} \right)^{P(m,n)}
 \underset{(m,n)\in M_{1}}{\prod} \left( \overline{a_{m,n}}r^{ m+n} \right)^{Q(m,n)}
\end{equation}

$$\times \underset{(m,n)\in M_{2}}{\prod} \left( a_{m,n}r^{ m -  n} \right)^{P(m,n)}
 \underset{(m,n)\in M_{2}}{\prod} \left( \overline{a_{m,n}}r^{ m-n} \right)^{Q(m,n)}
$$
$$ \times e^{(-2 \pi i\theta  \underset{(m,n)\in M_{1}}{\sum} Q(m,n)mn)}
e^{(-2 \pi i\theta  \underset{(m,n)\in M_{2}}{\sum} Q(m,n)mn)}  
$$
$$
\times \tau \left(\underset{\sigma \in \mathcal{S}(M_{P,Q})} {\sum}
  \prod_{i=1}^{k} U^{\sigma_{1}(m_{i}^{P,Q},n_{i}^{P,Q})} V^{\sigma_{2}(m_{i}^{P,Q},n_{i}^{P,Q})}
  \right),
$$
where $H_{k}$ is the set of all pairs $(P,Q)$ such that $$P:M_{1}\cup M_{2} \longrightarrow \mathbb{Z}_{0}^{+},$$  $$Q:M_{1}\cup M_{2} \longrightarrow \mathbb{Z}_{0}^{+},$$ 
\begin{equation}\label{8}
\underset{(m,n) \in M_{1} \cup M_{2}}{\sum} P(m,n)n =\underset{(m,n) \in M_{1} \cup M_{2}}{\sum} Q(m,n)n,
\end{equation}
\begin{equation}\label{9}
\underset{(m,n) \in M_{1} \cup M_{2}}{\sum} P(m,n)m =\underset{(m,n) \in M_{1} \cup M_{2}}{\sum} Q(m,n)m,
\end{equation}

\begin{equation}\label{10}
\underset{(m,n) \in M_{1}\cup M_{2}}{\sum} P(m,n) + \underset{(m,n) \in M_{1}\cup M_{2}}{\sum} Q(m,n)=k
\end{equation}
and
$(M_{P,Q}, [P,Q])$ is a multiset defined by $$
M_{P,Q}=M_{P}^{1,2} \cup M_{Q}^{-1,-2}
$$
where
$$
M_{P}^{1,2}=\left \{ (m,n) \in M_{1} \cup M_{2} :\, P(m,n) \neq  0 \right \},
$$

$$
M_{Q}^{-1,-2}=\left \{ (m,n) \in M :\,(m,n) \notin M_{1} \cup M_{2},  Q(-m,-n) \neq  0 \right \}
$$ 
and
$$
[P,Q]:M_{P,Q} \longrightarrow \mathbb{Z}_{0}^{+},
$$ 
is defined by
$$[P,Q](m,n)=\left\{\begin{matrix}
P(m,n) &  (m,n) \in M_{P}^{1,2} \\
Q(-m,-n) & (m,n) \in M_{Q}^{-1,-2}
\end{matrix}\right.
$$
Also, $(m_{i}^{P,Q},n_{i}^{P,Q})$ for $i=1,2, \cdots,k$ is a labeling for elements of $M_{P,Q}$.

Now we show that for $(P,Q) \in H_{k}$,
\begin{equation}\label{qlm}
\tau \left(\underset{\sigma \in \mathcal{S}(M_{P,Q})} {\sum}
  \prod_{i=1}^{k} U^{\sigma_{1}(m_{i}^{P,Q},n_{i}^{P,Q})} V^{\sigma_{2}(m_{i}^{P,Q},n_{i}^{P,Q})}
  \right)
\end{equation}
  $$=e^{\pi i \theta \underset{(m,n) \in M_{1}\cup M_{2}}{\sum} (P(m,n)+ Q(m,n))mn}
B_{P,Q},$$
where $B_{P,Q}$ is a real number which depends on $P$ and $Q$. Indeed,
for $(P,Q) \in H_{k}$ and $\sigma \in \mathcal{S}(M_{P,Q})$ we have (for simplicity we drop the superscript $P,Q$)
\begin{equation} \label{wwe}
\prod_{i=1}^{k} U^{\sigma_{1}(m_{i},n_{i})} V^{\sigma_{2}(m_{i},n_{i})}= e^{2 \pi i \theta B_{\sigma} } U^{ (\sum_{i=1}^{k}\sigma_{1}(m_{i},n_{i}))}V^{(\sum_{i=1}^{k}\sigma_{2}(m_{i},n_{i}))},
\end{equation}
where 
$$
B_{\sigma}=-\sigma_{1}(m_{2},n_{2})\sigma_{2}(m_{1},n_{1})
$$
$$
-\sigma_{1}(m_{3},n_{3})\left [\sigma_{2}(m_{1},n_{1})
+\sigma_{2}(m_{2},n_{2}) \right]
$$ 
$$
-\sigma_{1}(m_{4},n_{4})
\left[ \sigma_{2}(m_{1},n_{1})+ 
\sigma_{2}(m_{2},n_{2})+
\sigma_{2}(m_{3},n_{3}) \right]- \cdots
$$
$$
-\sigma_{1}(m_{k},n_{k})
\left[ \sigma_{2}(m_{1},n_{1})+ 
\sigma_{2}(m_{2},n_{2})+
\sigma_{2}(m_{3},n_{3})
+ \cdots 
+
\sigma_{2}(m_{k-1},n_{k-1}) \right].
$$
Since $(P,Q) \in H_{k}$, (\ref{8}) and (\ref{9}) implies

\begin{equation} \label{qwe}
\sum_{i=1}^{k}\sigma_{1}(m_{i},n_{i})=0
\end{equation}
and
\begin{equation} \label{qwr}
\sum_{i=1}^{k}\sigma_{2}(m_{i},n_{i})=0.
\end{equation}

So using (\ref{wwe}),
\begin{equation} \label{wsa}
\tau (\prod_{i=1}^{k} U^{\sigma_{1}(m_{i},n_{i})} V^{\sigma_{2}(m_{i},n_{i})})= e^{2 \pi i \theta B_{\sigma}}.
\end{equation}

Let

$$
A_{\sigma}=\sigma_{1}(m_{1},n_{1})\left [\sigma_{2}(m_{1},n_{1})
+\sigma_{2}(m_{2},n_{2})+ \cdots 
\sigma_{2}(m_{k},n_{k}) \right]
$$ 
$$+\sigma_{1}(m_{2},n_{2})\left [\sigma_{2}(m_{2},n_{2})+ \sigma_{2}(m_{3},n_{3})+ \cdots 
\sigma_{2}(m_{k},n_{k}) \right]
$$
$$
+\sigma_{1}(m_{3},n_{3})\left [\sigma_{2}(m_{3},n_{3})+ \sigma_{2}(m_{4},n_{4})+ \cdots 
\sigma_{2}(m_{k},n_{k}) \right]+\cdots
$$
$$
+\sigma_{1}(m_{k-1},n_{k-1})\left [\sigma_{2}(m_{k-1},n_{k-1})+ \sigma_{2}(m_{k},n_{k}) \right]
$$
$$
+\sigma_{1}(m_{k},n_{k})\sigma_{2}(m_{k},n_{k}).
$$
Using (\ref{qwe}) and (\ref{qwr}), one can easily check that $B_{\sigma}-A_{\sigma}=0$. So 
\begin{equation} \label{nmh}
B_{\sigma}=\dfrac{1}{2}(B_{\sigma}+A_{\sigma}).
\end{equation}

We also set 
$$D_{\sigma}= \sum_{j=2}^{k}\sum_{i=1}^{j-1} \left[\sigma_{1}(m_{i},n_{i})\sigma_{2}(m_{j},n_{j})-\sigma_{1}(m_{j},n_{j})\sigma_{2}(m_{i},n_{i})\right].$$
One can easily see that 
\begin{equation} \label{dsi}
D_{\sigma}= \sum_{j=1}^{k-1}\sum_{i=j+1}^{k} \left[\sigma_{1}(m_{j},n_{j})\sigma_{2}(m_{i},n_{i})-\sigma_{1}(m_{i},n_{i})\sigma_{2}(m_{j},n_{j})\right].
\end{equation}
Then we see that
$$
B_{\sigma}+A_{\sigma}=D_{\sigma}+\sum_{i=1}^{k}\sigma_{1}(m_{i},n_{i})\sigma_{2}(m_{i},n_{i}) $$
$$=D_{\sigma}+\sum_{(m,n)\in M_{P,Q}} [P,Q](m,n)mn
$$
$$
=D_{\sigma}+\sum_{(m,n)\in M_{P,Q}^{1,2}}P(m,n)mn+\sum_{(m,n)\in M_{P,Q}^{-1,-2}}Q(-m,-n)mn
$$
$$=D_{\sigma}+\sum_{(m,n)\in M_{1} \cup M_{2}}P(m,n)mn+\sum_{(-m,-n)\in M_{1} \cup M_{2}}Q(-m,-n)mn
$$
$$=D_{\sigma}+\sum_{(m,n)\in M_{1} \cup M_{2}}P(m,n)mn+\sum_{(m,n)\in M_{1} \cup M_{2}}Q(m,n)mn
$$
Therefore, by (\ref{nmh}) we have
$$
B_{\sigma}=\dfrac{1}{2}\left [D_{\sigma}+\sum_{(m,n)\in M_{1} \cup M_{2}}(P(m,n)+Q(m,n))mn \right].
$$
Then regarding (\ref{wsa}) we have 
$$
\tau (\prod_{i=1}^{k} U^{\sigma_{1}(m_{i},n_{i})} V^{\sigma_{2}(m_{i},n_{i})})= e^{2 \pi i \theta B_{\sigma}}$$
$$
=e^{\pi i \theta \underset{(m,n)\in M_{1}\cup M_{2}}{\sum}(P(m,n)+Q(m,n))mn} e^{\pi i \theta D_{\sigma}}
$$
Now if we define 
$$
B_{P,Q}=\underset{\sigma \in \mathcal{S}(M_{P,Q})} {\sum}e^{\pi i \theta D_{\sigma}},
$$
we see that 
$$
\tau \left(\underset{\sigma \in \mathcal{S}(M_{P,Q})} {\sum}
  \prod_{i=1}^{k} U^{\sigma_{1}(m_{i}^{P,Q},n_{i}^{P,Q})} V^{\sigma_{2}(m_{i}^{P,Q},n_{i}^{P,Q})}
  \right).
$$
$$=e^{\pi i \theta \underset{(m,n) \in M_{1}\cup M_{2}}{\sum} (P(m,n)+ Q(m,n))mn}
B_{P,Q}.$$
So we have proved (\ref{qlm}). Now we 
will show that $B_{P,Q}$ is a real number. Indeed, for $\sigma \in \mathcal{S}(M_{P,Q})$, we define $\beta _{\sigma} \in \mathcal{S}(M_{P,Q})$ by $$\beta _{\sigma}(m_i,n_i)=\sigma(m_{k-i+1},n_{k-i+1})\, \,\,\,\, i=1,2, \cdots, k.$$
Then we have
$$
D_{\beta_{\sigma}}= \sum_{j=2}^{k}\sum_{i=1}^{j-1} \left[{\beta_{\sigma}}_{1}(m_{i},n_{i}){\beta_{\sigma}}_{2}(m_{j},n_{j})-{\beta_{\sigma}}_{1}(m_{j},n_{j}){\beta_{\sigma}}_{2}(m_{i},n_{i})\right]
$$
$$=\sum_{j=2}^{k}\sum_{i=1}^{j-1} \sigma_{1}(m_{k-i+1},n_{k-i+1})\sigma_{2}(m_{k-j+1},n_{k-j+1})$$
$$-\sum_{j=2}^{k}\sum_{i=1}^{j-1}\sigma_{1}(m_{k-j+1},n_{k-j+1})\sigma_{2}(m_{k-i+1},n_{k-i+1})$$
$$= \sum_{t=k-1}^{1}\sum_{s=k}^{t+1} \left[\sigma_{1}(m_{s},n_{s})\sigma_{2}(m_{t},n_{t})-\sigma_{1}(m_{t},n_{t})\sigma_{2}(m_{s},n_{s})\right]$$
$$= \sum_{t=1}^{k-1}\sum_{s=t+1}^{k} \left[\sigma_{1}(m_{s},n_{s})\sigma_{2}(m_{t},n_{t})-\sigma_{1}(m_{t},n_{t})\sigma_{2}(m_{s},n_{s})\right]$$
$$= -\sum_{t=1}^{k-1}\sum_{s=t+1}^{k} \left[\sigma_{1}(m_{t},n_{t})\sigma_{2}(m_{s},n_{s})-\sigma_{1}(m_{s},n_{s})\sigma_{2}(m_{t},n_{t})\right]=-D_{\sigma},$$
where in the last equality we have used (\ref{dsi}). Now we have
$$
B_{P,Q}=\underset{\sigma \in \mathcal{S}(M_{P,Q})} {\sum}e^{\pi i \theta D_{\sigma}}=\dfrac{1}{2}(\underset{\sigma \in \mathcal{S}(M_{P,Q})} {\sum}e^{\pi i \theta D_{\sigma}}+\underset{\sigma \in \mathcal{S}(M_{P,Q})} {\sum}e^{\pi i \theta D_{\sigma}})
$$
$$
=\dfrac{1}{2}(\underset{\sigma \in \mathcal{S}(M_{P,Q})} {\sum}e^{\pi i \theta D_{\sigma}}+\underset{\sigma \in \mathcal{S}(M_{P,Q})} {\sum}e^{\pi i \theta D_{\beta_{\sigma}}})
$$
$$
=\dfrac{1}{2}(\underset{\sigma \in \mathcal{S}(M_{P,Q})} {\sum}e^{\pi i \theta D_{\sigma}}+\underset{\sigma \in \mathcal{S}(M_{P,Q})} {\sum}e^{-\pi i \theta D_{\sigma}})=\dfrac{1}{2}\underset{\sigma \in \mathcal{S}(M_{P,Q})} {\sum}(e^{\pi i \theta D_{\sigma}}+e^{-\pi i \theta D_{\sigma}})
$$
So $B_{P,Q}$ is real.

Now using (\ref{mmm}), we see
$$
\tau (x_{r}^{k})=\sum_{(P,Q) \in H_{k}}  \underset{(m,n)\in M_{1}}{\prod} \left( a_{m,n}r^{ m +  n} \right)^{P(m,n)}
 \underset{(m,n)\in M_{1}}{\prod} \left( \overline{a_{m,n}}r^{ m+n} \right)^{Q(m,n)}
$$

$$\underset{(m,n)\in M_{2}}{\prod} \left( a_{m,n}r^{ m -  n} \right)^{P(m,n)}
 \underset{(m,n)\in M_{2}}{\prod} \left( \overline{a_{m,n}}r^{ m-n} \right)^{Q(m,n)}
$$
$$ e^{-2 \pi i\theta  \underset{(m,n)\in M_{1} \cup M_{2}}{\sum} Q(m,n)mn}
e^{\pi i \theta \underset{(m,n) \in M_{1}\cup M_{2}}{\sum} (P(m,n)+ Q(m,n))mn}
B_{P,Q}$$

Then we have
$$
\tau (x_{r}^{k})=\sum_{(P,Q) \in H_{k}} 
r^{\left(\underset{(m,n) \in M_{1}}{\sum} \left( P(m,n)+Q(m,n) \right) (m+n)+ \underset{(m,n) \in M_{2}}{\sum} \left( P(m,n)+Q(m,n) \right) (m-n)\right)} 
$$
$$
e^{ \pi i\theta  \underset{(m,n)\in M_{1} \cup M_{2}}{\sum} P(m,n)mn}
e^{- \pi i\theta  \underset{(m,n)\in M_{1} \cup M_{2}}{\sum} Q(m,n)mn}
$$
$$
\underset{(m,n)\in M_{1} \cup M_{2}}{\prod} a_{m,n}^{P(m,n)} 
 \underset{(m,n)\in M_{1} \cup M_{2}}{\prod}  \overline{a_{m,n}}^{Q(m,n)} B_{P,Q}
$$
$$
=\sum_{(P,Q) \in H_{k}}
r^{\left(\underset{(m,n) \in M_{1} \bigcup M_{2} }{\sum} P(m,n)m+\underset{(m,n) \in M_{1} \bigcup M_{2} }{\sum} Q(m,n)m \right)}
$$
$$ r^{ \left( \underset{(m,n) \in M_{1}}{\sum} P(m,n)n-\underset{(m,n) \in M_{2}}{\sum} Q(m,n)n +\underset{(m,n) \in M_{1}}{\sum} Q(m,n)n-\underset{(m,n) \in M_{2}}{\sum} P(m,n)n\right)} 
$$
$$
e^{ \pi i\theta  \underset{(m,n)\in M_{1} \cup M_{2}}{\sum} P(m,n)mn}
e^{- \pi i\theta  \underset{(m,n)\in M_{1} \cup M_{2}}{\sum} Q(m,n)mn}
$$
$$
\underset{(m,n)\in M_{1} \cup M_{2}}{\prod} a_{m,n}^{P(m,n)} 
 \underset{(m,n)\in M_{1} \cup M_{2}}{\prod}  \overline{a_{m,n}}^{Q(m,n)} B_{P,Q}
$$
So
\begin{equation} \label{aaa}
\tau (x_{r}^{k})=\sum_{\underset{l\geq 0, s\geq 0}{l+s=2}}^{\infty} r^{2(l+s)} 
 \sum_{(P,Q) \in G_{l,s}}
e^{ \pi i\theta  \underset{(m,n)\in M_{1} \cup M_{2}}{\sum} P(m,n)mn}
\end{equation}
$$
\times e^{- \pi i\theta  \underset{(m,n)\in M_{1} \cup M_{2}}{\sum} Q(m,n)mn}\underset{(m,n)\in M_{1} \cup M_{2}}{\prod} a_{m,n}^{P(m,n)} 
 \underset{(m,n)\in M_{1} \cup M_{2}}{\prod}  \overline{a_{m,n}}^{Q(m,n)} B_{P,Q}
$$
 where $G_{l,s}$ is the set of all pairs $(P,Q)$ such that $$P:M_{1}\cup M_{2} \longrightarrow \mathbb{Z}_{0}^{+},$$  $$Q:M_{1}\cup M_{2} \longrightarrow \mathbb{Z}_{0}^{+},$$ and
$$
\underset{(m,n) \in M_{1}\cup M_{2}}{\sum} P(m,n) + \underset{(m,n) \in M_{1}\cup M_{2}}{\sum} Q(m,n)=k,
$$
$$
\underset{(m,n) \in M_{1}\cup M_{2}}{\sum} P(m,n)m = \underset{(m,n) \in M_{1} \cup M_{2}}{\sum} Q(m,n)m =l,
$$
$$\underset{(m,n) \in M_{1}}{\sum} P(m,n)n- \underset{(m,n) \in M_{2}}{\sum} Q(m,n)n$$
$$=\underset{(m,n) \in M_{1}}{\sum} Q(m,n)n- \underset{(m,n) \in M_{2}}{\sum} P(m,n)n=s.
$$
One should note that in (\ref{aaa}) $(l+s)$ starts from $2$. Here we shall show why that is the case:
for $(m,n) \in M_{1}$ if $m=0$, then $n \neq0$. So we have
\begin{equation}
\underset{(m,n) \in M_{1}}{\sum} P(m,n)\leqslant \underset{(m,n) \in M_{1}}{\sum} P(m,n)m+\underset{(m,n) \in M_{1}}{\sum} P(m,n)n.
\end{equation}
Similarly we have 
\begin{equation} \label{bbb}
\underset{(m,n) \in M_{2}}{\sum} P(m,n)\leqslant \underset{(m,n) \in M_{2}}{\sum} P(m,n)m-\underset{(m,n) \in M_{2}}{\sum} P(m,n)n,
\end{equation}
\begin{equation} \label{ccc}
\underset{(m,n) \in M_{1}}{\sum} Q(m,n)\leqslant \underset{(m,n) \in M_{1}}{\sum} Q(m,n)m+\underset{(m,n) \in M_{1}}{\sum} Q(m,n)n,
\end{equation}
\begin{equation} \label{ddd}
\underset{(m,n) \in M_{2}}{\sum} Q(m,n)\leqslant \underset{(m,n) \in M_{2}}{\sum} Q(m,n)m-\underset{(m,n) \in M_{2}}{\sum} Q(m,n)n.
\end{equation}
So $$k=\underset{(m,n) \in M_{1}\cup M_{2}}{\sum} P(m,n) + \underset{(m,n) \in M_{1}\cup M_{2}}{\sum} Q(m,n)\leqslant$$
$$\underset{(m,n) \in M_{1}\cup M_{2}}{\sum} P(m,n)m + \underset{(m,n) \in M_{1} \cup M_{2}}{\sum} Q(m,n)m $$
$$ +\underset{(m,n) \in M_{1}}{\sum} P(m,n)n- \underset{(m,n) \in M_{2}}{\sum} Q(m,n)n$$
$$+\underset{(m,n) \in M_{1}}{\sum} Q(m,n)n- \underset{(m,n) \in M_{2}}{\sum} P(m,n)n$$
$$=2(s+l)$$
So for a fixed $k$, $\dfrac{k}{2} \leqslant l+s$ and since $k$ is at least $3$, $l+s \geqslant 2$.
 
Let $ \tilde{G}_{l,s}$ be the set of all pairs $(\tilde{P},\tilde{Q})$ such that $$\tilde{P}:M\longrightarrow \mathbb{Z}_{0}^{+},$$  $$\tilde{Q}:M \longrightarrow \mathbb{Z}_{0}^{+},$$ and
$$
\underset{(m,n) \in M_{1}\cup M_{2}}{\sum} \tilde{P}(m,n) + \underset{(m,n) \in M_{1}\cup M_{2}}{\sum} \tilde{Q}(m,n)=k,
$$
$$
\underset{(m,n) \in M_{1}\cup M_{2}}{\sum} \tilde{P}(m,n)m = \underset{(m,n) \in M_{1} \cup M_{2}}{\sum} \tilde{Q}(m,n)m =l,
$$
$$\underset{(m,n) \in M_{1}}{\sum} \tilde{P}(m,n)n- \underset{(m,n) \in M_{2}}{\sum} \tilde{Q}(m,n)n$$
$$=\underset{(m,n) \in M_{1}}{\sum} \tilde{Q}(m,n)n- \underset{(m,n) \in M_{2}}{\sum} \tilde{P}(m,n)n=s.
$$
There exist a one to one correspondence between $ \tilde{G}_{l,s}$ and $ G_{l,s}$.
Infact, for $(P,Q)\in G_{l,s}$, we can define 
$$\tilde{P}(m,n)=\left\{\begin{matrix}
P(m,n) & (m,n)\in M_{1}\cup M_{2} \\ 
Q(-m,-n) & (m,n)\notin M_{1}\cup M_{2}
\end{matrix}\right.$$
and 
$$\tilde{Q}(m,n)=\left\{\begin{matrix}
Q(m,n) & (m,n)\in M_{1}\cup M_{2} \\ 
P(-m,-n) & (m,n)\notin M_{1}\cup M_{2}
\end{matrix}\right.$$

Using this correspondence and the fact that $B_{P,Q}=B_{\tilde{P}\restriction_{M_{1}\cup M_{2}} ,\tilde{Q}\restriction_{M_{1}\cup M_{2}}}$ in (\ref{aaa}), we have
\begin{equation} \label{tyu}
\tau (x_{r}^{k})=\sum_{\underset{l\geq 0, s\geq 0}{l+s=2}}^{\infty} r^{2(l+s)} 
 \sum_{(\tilde{P},\tilde{Q}) \in \tilde{G}_{l,s}}
e^{ \pi i\theta  \underset{(m,n)\in M_{1} \cup M_{2}}{\sum} \tilde{P}(m,n)mn}
\end{equation}
$$
\times e^{- \pi i\theta  \underset{(m,n)\in M_{1} \cup M_{2}}{\sum} \tilde{Q}(m,n)mn}\underset{(m,n)\in M_{1} \cup M_{2}}{\prod} a_{m,n}^{\tilde{P}(m,n)} 
 \underset{(m,n)\in M_{1} \cup M_{2}}{\prod}  \overline{a_{m,n}}^{\tilde{Q}(m,n)} B_{\tilde{P},\tilde{Q}}.
$$

Now If we can decompose 
$B_{\tilde{P},\tilde{Q}}$ into two terms $B_{\tilde{P}}$ and  $B_{\tilde{Q}}$, i.e.
\begin{equation} \label{er}
B_{\tilde{P},\tilde{Q}}=B_{\tilde{P}}B_{\tilde{Q}}
\end{equation} 
such that $B_{\tilde{P}}$ and $B_{\tilde{Q}}$ depend only respectively on and $\tilde{P}$ and $\tilde{Q}$,
then we can easily continue the proof of Theorem \ref{man}. Indeed, If (\ref{er}) holds, for a function $$P:M \longrightarrow \mathbb{Z}_{0}^{+},$$  we can define 
\begin{equation} \label{C(P)}
D(P)= e^{\pi i \theta\underset{(m,n) \in M_{1}\cup M_{2}}{\sum} P(m,n)mn} \underset{(m,n)\in M_{1}\cup M_{2}}{\prod} \left( -a_{m,n} \right)^{P(m,n)} B_{P},
\end{equation} 
and the rest would be much similar to the proof of Theorem \ref{man}.


\end{document}